\documentclass{amsart}
\usepackage{graphicx,tikz,color}			
\usetikzlibrary{shapes, arrows, calc, arrows.meta, fit, positioning}
\usepackage[all, cmtip]{xy}
\usepackage{ytableau}
\usepackage{hyperref}
\usepackage{amssymb}

\usetikzlibrary{decorations.markings}

\theoremstyle{theorem}
\newtheorem{theorem}{Theorem}
\newtheorem{proposition}{Proposition}
\newtheorem{lemma}{Lemma}
 
\theoremstyle{definition}

\newtheorem*{remark}{Remark}

\newcommand{\Z}{\mathbb{Z} }
\newcommand{\Mod}[1]{\ (\mathrm{mod}\ #1)}

\begin{document}

\title{Sliding Block Puzzles with a Twist:  On Segerman's 15+4 Puzzle}
\markright{Sliding Block Puzzles with a Twist}
\author{Patrick Garcia, Angela Hanson, David Jensen, and Noah Owen}

\maketitle

\begin{abstract}
Segerman's 15+4 puzzle is a hinged version of the classic 15-puzzle, in which the tiles rotate as they slide around.  In 1974, Wilson classified the groups of solutions to sliding block puzzles.  We generalize Wilson's result to puzzles like the 15+4 puzzle, where the tiles can rotate, and the sets of solutions are subgroups of the generalized symmetric groups.  Aside from two exceptional cases, we see that the group of solutions to such a puzzle is always either the entire generalized symmetric group or one of two special subgroups of index two.
\end{abstract}

\section{Introduction}
\label{Sec:Intro}

\subsection{Sliding Block Puzzles}
The classic 15-puzzle, pictured in Figure~\ref{Fig:15Puzz}, is a $4 \times 4$ grid containing 15 tiles, labeled with the numbers 1 through 15, and one empty square.  You can move the tiles around by sliding an adjacent tile into the empty square, and the goal is to arrange the tiles in order from 1 to 15 via a sequence of these moves.  A natural question is: given a starting permutation of the tiles, is it always possible to solve the puzzle?  If not, which permutations are solvable?  In 1879, Johnson and Story showed that the 15-puzzle cannot be solved if the starting permutation is odd \cite{JohnsonStory}.  Thus, for example, if you were to pop out two of the tiles, switch their places, and pop them back in again, the resulting puzzle would be impossible to solve.

\begin{figure}[h]
\begin{center}
\begin{ytableau}
1 & 2 & 3 & 4 \\
5 & 6  & 7 & 8\\
9 & 10 & 11 & 12 \\
13 & 14 & 15 & {}\\
\end{ytableau}

\caption{The 15-puzzle}
\label{Fig:15Puzz}
\end{center}
\end{figure}

In 1974, Wilson considered more general sliding block puzzles on arbitrary graphs, where the vertices represent the possible positions of the tiles and there is an edge between two vertices if the corresponding tiles are adjacent \cite{Wilson}.  If the graph has $n+1$ vertices, then the set of solvable permutations in which the empty vertex is fixed forms a subgroup of the symmetric group $S_n$.  Wilson's main result is a classification of these subgroups.

\begin{theorem} \cite[Theorem~2]{Wilson}
\label{Thm:Wilson}
Let $\Gamma$ be a simple, 2-vertex connected graph on $n+1$ vertices that is not a cycle.  Then the group of solvable permutations is:
\begin{itemize}
\item  $A_n$ if $\Gamma$ is bipartite,
\item $\mathrm{PGL} (2,5)$ if $\Gamma$ is the graph $\Theta_7$ pictured in Figure~\ref{Fig:Theta7}, and
\item $S_n$ otherwise.
\end{itemize}
\end{theorem}

For example, the graph corresponding to the 15-puzzle is bipartite, so the set of solvable permutations is the alternating group $A_{15}$.  Here, $\mathrm{PGL} (2,5) \cong S_5$ is the ``exotic'' subgroup of $S_6$ isomorphic to $S_5$.  It acts on the vertices of $\Theta_7$, labeled as in Figure~\ref{Fig:Theta7}, by fractional linear transformations.

\begin{figure}[h]
		\begin{center}
                \begin{tikzpicture}[main/.style = {fill = black}]
    
                    \draw[main] (1,0) circle (4pt);
                    \draw[main] (-1,0) circle (4pt);
                    \draw[main] (0,0) circle (4pt);
                    \draw[main] (0.5,0.87) circle (4pt);
                    \draw[main] (-0.5, 0.87) circle (4pt);
                    \draw[main] (-0.5,-0.87) circle (4pt);
                    \draw[main] (0.5, -0.87) circle (4pt);
                    
                    \draw (-0.7,1.2) node{0};
                    \draw (0.7,1.2) node{1};
                    \draw (1.5,0) node{2};
                    \draw (0.7,-1.2) node{3};
                    \draw (-0.7,-1.2) node{4};
                    \draw (-1.5,0) node{$\infty$};
    
                    \draw (1,0) -- (0.5,0.87)  -- (-0.5, 0.87) -- (-1,0) -- (-0.5,-0.87) -- (0.5, -0.87) -- (1,0);
                    \draw (-1,0) -- (0, 0) -- (1,0);
    
            \end{tikzpicture}
            \caption{The graph $\Theta_7$}
            \label{Fig:Theta7}
            \end{center}
\end{figure}
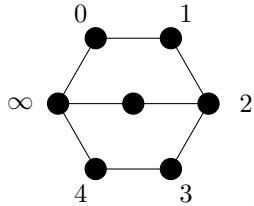

\begin{remark}
The assumption that $\Gamma$ is simple in Theorem~\ref{Thm:Wilson} is not necessary.  If a pair of vertices in $\Gamma$ have multiple edges between them, then sliding a tile along any of these edges results in the same permutation of the tiles.  Thus, if $\Gamma$ has multi-edges, the group of solvable permutations is the same as that of the simple graph obtained by replacing each multi-edge in $\Gamma$ by a single edge.
\end{remark}

\subsection{The $15+4$ Puzzle}
In 2021, Segerman created the $15+4$ puzzle, pictured in Figure~\ref{Fig:15+4}.  Much like the 15-puzzle, Segerman's puzzle consists of one empty vertex and 19 tiles, which can be moved by sliding an adjacent tile into the empty spot.  The twist is that moving the tiles around can result in some of them becoming rotated.  For example, in Figure~\ref{Fig:PermutedPuzzle}, the tiles labeled 6, 7,10, and 16 are rotated 90 degrees from their starting positions.  The goal is not only to get the tiles in the correct order, but with the correct rotations as well.  More details on the $15+4$ puzzle are provided in Dr. Segerman's YouTube video: 
\[
\mbox{\url{https://www.youtube.com/watch?v=Hc3yfuXiWe0}.}
\]

\begin{figure}[h]
\begin{center}
\includegraphics[scale=0.2]{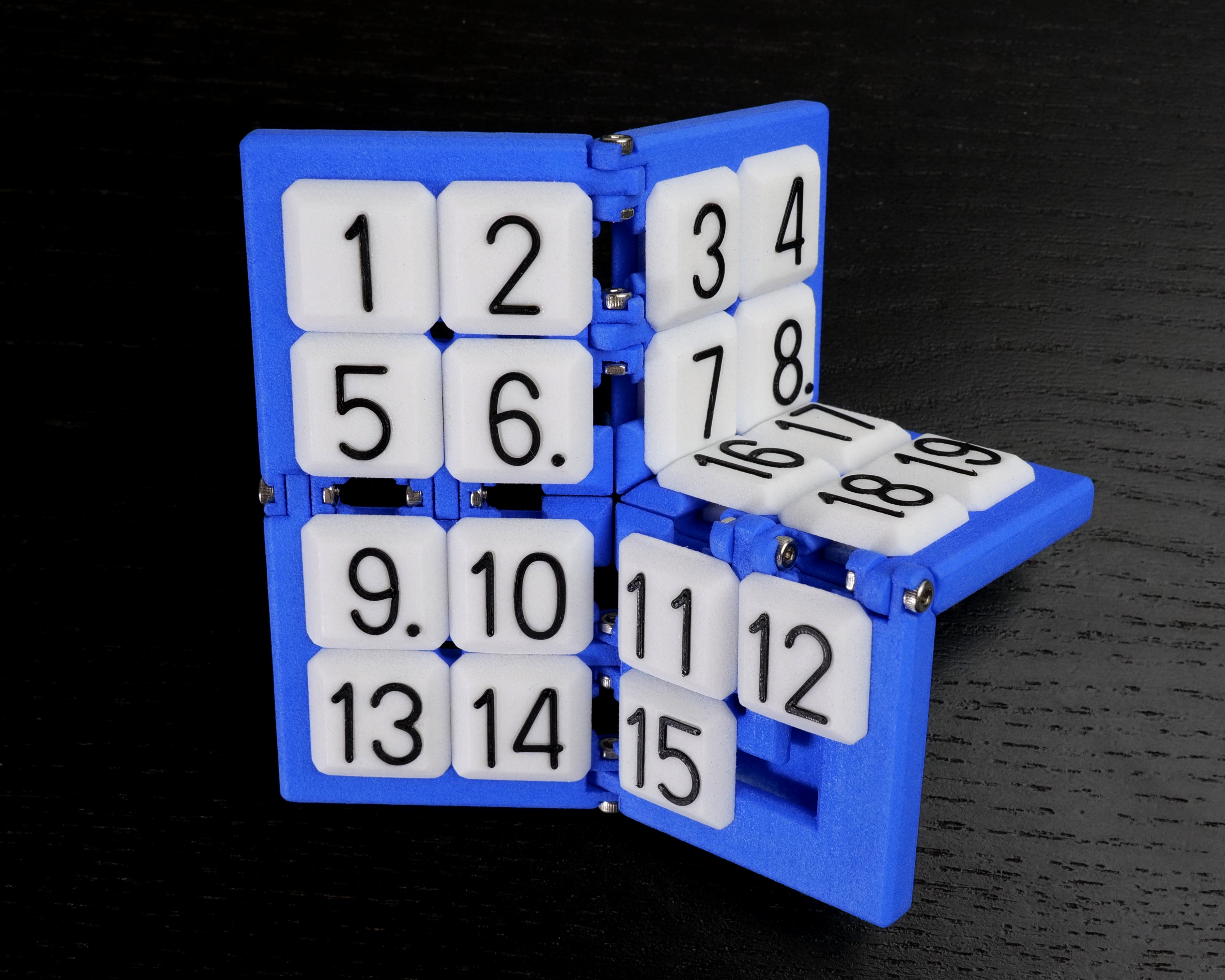}
\caption{Segerman's $15+4$ Puzzle (Photo courtesy of Henry Segerman)}
\label{Fig:15+4}
\end{center}
\end{figure}

\begin{figure}[h]
\begin{center}
\includegraphics[scale=0.27]{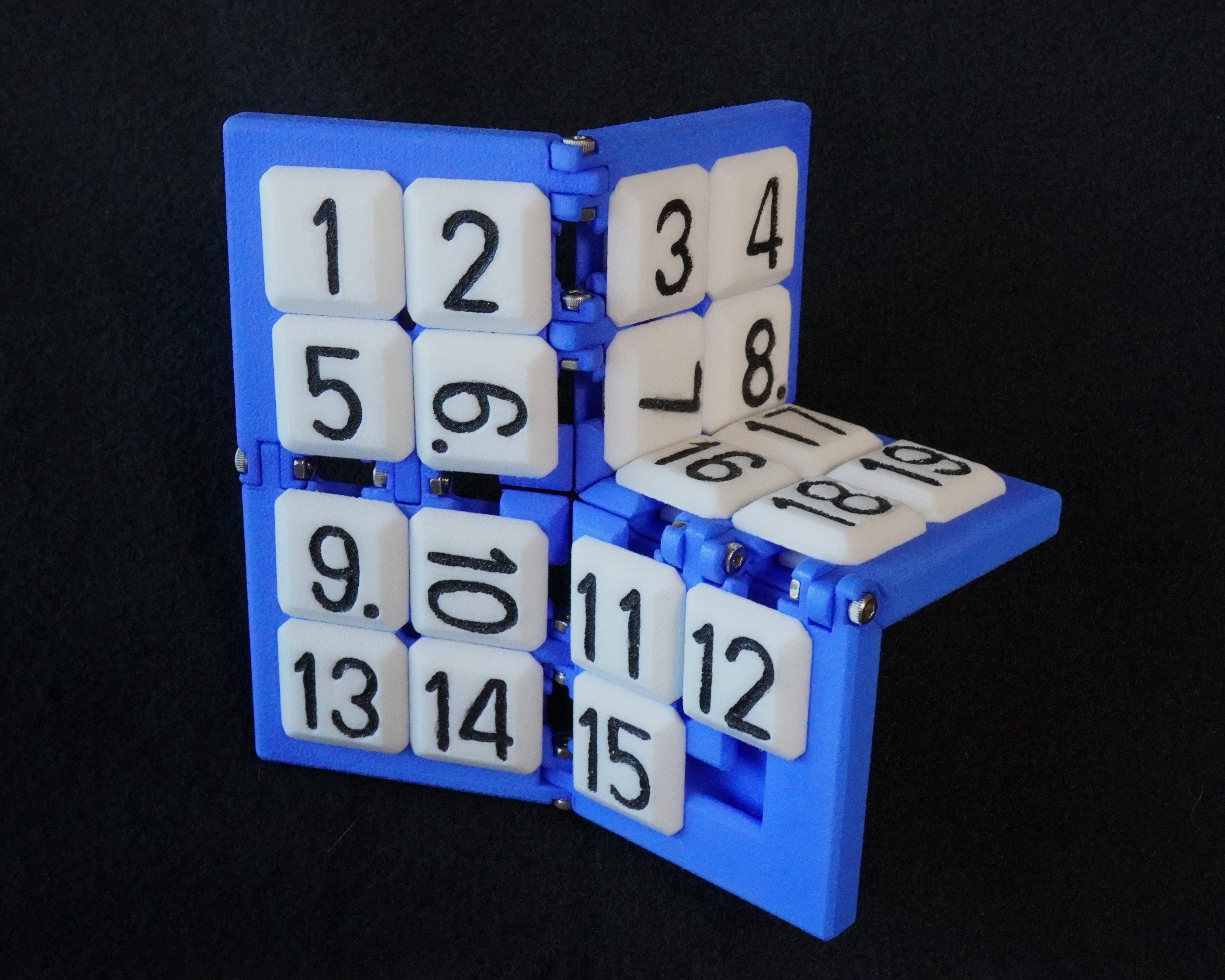}
\caption{A solvable state of the $15+4$ puzzle, in which some of the tiles are rotated}
\label{Fig:PermutedPuzzle}
\end{center}
\end{figure}

The goal of this note is to generalize Wilson's result to sliding block puzzles in which the tiles rotate as they slide around.  We assume throughout that each tile has $m$ sides, so it can rotate $m$ times.  Each position of the puzzle determines an ordered pair $(\vec{x},\sigma)$, where $\sigma \in S_n$ is a permutation and $\vec{x} \in (\Z/m\Z)^n$ records the rotations of the tiles.  As we explain in greater detail in Section~\ref{Sec:GenSym}, the set of such ordered pairs forms a subgroup of the generalized symmetric group $S(m,n)$.

Figure~\ref{Fig:Twist15+4} depicts a graphical representation of Segerman's 15+4 puzzle.  Here, each tile has 4 sides.  If a tile slides in either direction across one of the solid lines, then it does not rotate.  If it slides across one of the dashed lines from left to right, it rotates 90 degrees, and if it slides from right to left, it rotates 90 degrees in the other direction.  Because the graph is not bipartite, by Theorem~\ref{Thm:Wilson}, it is possible to obtain every permutation of the tiles from any starting position.  As we will see in Theorem~\ref{Thm:MainThm} below, however, it is not possible to obtain every possible ordered pair of permutations and rotations.  In particular, in any odd permutation of the tiles, some of the tiles will be non-trivially rotated.

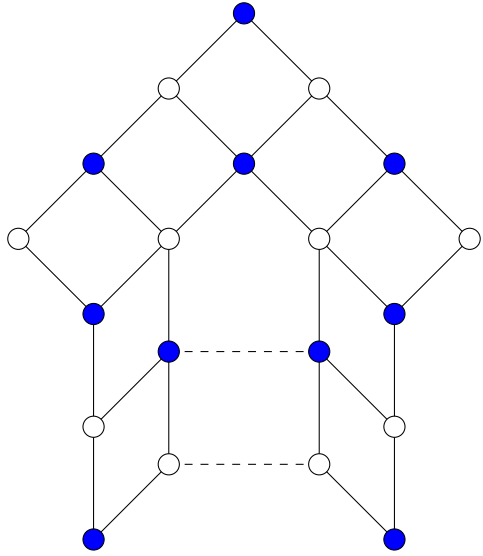
\begin{figure}[h]
		\begin{center}
                \begin{tikzpicture}[main/.style = {fill = black}]
                
                                    \draw (0,0) -- (-1,-1) -- (-2,-2) -- (-3,-3) -- (-2,-4) -- (-2,-4.5) -- (-2, -7) -- (-1, -6) -- (-1, -4.5) -- (-1, -3) -- (0,-2) -- (1,-3) -- (1, -4.5) -- (1, -6) -- (2, -7) -- (2, -5.5) -- (2, -4) -- (3,-3) -- (2,-2) -- (1, -1) -- (0,0);
    
                    \draw (-1,-1) -- (0, -2);
                    \draw (-2,-2) -- (-1,-3) -- (-2,-4);
                \draw (-2, -5.5) -- (-1, -4.5);
                \draw (-1,-4.5)[dashed] -- (1, -4.5);
    
                \draw (1,-1) -- (0,-2);
                \draw (2,-2) -- (1, -3) -- (2, -4);
                \draw (2, -5.5) -- (1, -4.5);
    
                \draw (-1,-6)[dashed] -- (1,-6);
    
                    \draw[black, fill=blue] (0,0) circle (4pt);
                    \draw[black, fill=white] (-1,-1) circle (4pt);
                    \draw[black, fill=white] (1,-1) circle (4pt);
                    \draw[black, fill=blue] (0,-2) circle (4pt);
                    \draw[black, fill=blue] (-2,-2) circle (4pt);
                    \draw[black, fill=blue] (2,-2) circle (4pt);
                    \draw[black, fill=white] (-3,-3) circle (4pt);
                    \draw[black, fill=white] (3,-3) circle (4pt);
                    \draw[black, fill=white] (-1,-3) circle (4pt);
                    \draw[black, fill=white] (1,-3) circle (4pt);
                    \draw[black, fill=blue] (-2,-4) circle (4pt);
                    \draw[black, fill=blue] (2,-4) circle (4pt);
                    \draw[black, fill=blue] (-1,-4.5) circle (4pt);
                    \draw[black, fill=blue] (1,-4.5) circle (4pt);
                    \draw[black, fill=white] (-2,-5.5) circle (4pt);
                    \draw[black, fill=white] (2,-5.5) circle (4pt);
                    \draw[black, fill=white] (-1, -6) circle (4pt);
                    \draw[black, fill=white] (1, -6) circle (4pt);
                    \draw[black, fill=blue] (2, -7) circle (4pt);
                    \draw[black, fill=blue] (-2,-7) circle (4pt);

            \end{tikzpicture}
            \caption{Twist graph representation of the $15+4$ puzzle}
            \label{Fig:Twist15+4}
            \end{center}
\end{figure}

\subsection{Twist Graphs}
Following this example, to keep track of how a puzzle rotates the tiles, we decorate each edge of the graph with an element of $\Z/m\Z$, which records how much a tile rotates when it slides along that edge.  More precisely, recall that a \emph{1-chain} $\gamma$ on a graph $\Gamma$  with coefficients in $\Z/m\Z$ consists of an element $\gamma_e \in \Z/m\Z$ for each oriented edge $e$ in $\Gamma$, subject to the condition that, if $e$ and $\overline{e}$ represent the same edge with opposite orientations, then $\gamma_{\overline{e}} = -\gamma_e$.  We define a \emph{twist graph} to be an ordered pair $(\Gamma, \gamma)$, where $\Gamma$ is a graph and $\gamma$ is a 1-chain on $\Gamma$ with coefficients in $\Z/m\Z$.  Returning to our example, Figure~\ref{Fig:Twist15+4} represents the $15+4$ puzzle by a twist graph $(\Gamma,\gamma)$.  Here, $m=4$, the solid lines represent edges $e \in E(\Gamma)$ with $\gamma_e \equiv 0 \Mod{4}$, and the dashed lines represent oriented edges $e$ with $\gamma_e \equiv 1 \Mod{4}$ when oriented from left to right.

We say that a twist graph $(\Gamma,\gamma)$ is \emph{twist bipartite} if $m$ is even and the vertices of $\Gamma$ can be colored using two colors such that:
\begin{enumerate}
\item  if $\gamma_e$ is even, then the endpoints of $e$ have different colors, and
\item  if $\gamma_e$ is odd, then the endpoints of $e$ have the same color.
\end{enumerate}
Equivalently, a twist graph is twist bipartite if the number of even edges in any cycle is even.  For example, the coloring of the vertices in Figure~\ref{Fig:Twist15+4} illustrates that the pictured twist graph is twist bipartite.  In particular, if one deletes the two dashed edges, then the resulting graph is bipartite, and in any 2-coloring of the vertices of this subgraph, the dashed edges connect vertices of the same color.

Our main result requires an assumption on the homology of the graph, which we now briefly explain.  The significance of this condition will be explored in greater detail at the end of Section~\ref{Sec:FundamentalGroup}.  A 1-chain $\omega$ on a graph $\Gamma$ is a \emph{1-cycle} if, for all vertices $v \in V(\Gamma)$, we have $\displaystyle\sum_{e \to v} \omega_e = 0$, where the sum is over all oriented edges with head $v$ \cite[p.~105]{Hatcher}.  The group of 1-cycles on $\Gamma$ with coefficients in $\Z/m\Z$ is called the \emph{first homology group} of $\Gamma$, and denoted $H_1 (\Gamma, \Z/m\Z)$.  Given a twist graph $(\Gamma,\gamma)$, there is a natural homomorphism $\varphi_{\gamma} \colon H_1 (\Gamma, \Z/m\Z) \to \Z/m\Z$ given by
\[
\varphi_{\gamma} (\omega) = \sum_{e \in E(\Gamma)} \gamma_e \cdot \omega_e  \Mod{m} .
\]
Throughout, we will restrict our attention to twist graphs for which the homomorphism $\varphi_{\gamma}$ is surjective.  In Lemmas~\ref{Lem:PuzzleToGraph} and~\ref{Lem:Gamma0Connected} of the next section, and the remarks immediately thereafter, we will see that we can always reduce to this case.  Note that if $m$ is even and $\varphi_{\gamma}$ is surjective, then there is at least one cycle in $\Gamma$ with an odd number of odd edges.  Thus, the twist graph $(\Gamma,\gamma)$ cannot be both bipartite and twist bipartite.  

Our main result is a classification of the possible permutations and rotations that can occur in sliding block puzzles where the tiles can rotate.

\begin{theorem}
\label{Thm:MainThm}
Let $(\Gamma,\gamma)$ be a 2-vertex connected twist graph with no loops and with $\varphi_{\gamma}$ surjective.  Suppose that $\Gamma$ is not a cycle, the graph $\Theta_5$ of Figure~\ref{Fig:Theta5}, the graph $\Theta_7$ of Figure~\ref{Fig:Theta7}, or any graph obtained from one of these three by replacing edges with multiple parallel edges.  Then the group of solvable ordered pairs $(\vec{x},\sigma) \in S(m,n)$ is:
\begin{itemize}
\item $\Big\{ (\vec{x},\sigma) \in S(m,n) \mbox{ } \vert \mbox{ } \sigma \in A_n \Big\}$ if $\Gamma$ is bipartite,
\item $\Big\{ (\vec{x},\sigma) \in S(m,n) \mbox{ } \vert \mbox{ } \sum_{i=1}^n x_i \equiv \mathrm{sign} (\sigma) \Mod{2} \Big\}$ if $m$ is even and $(\Gamma,\gamma)$ is twist bipartite, and
\item $S(m,n)$ otherwise.
\end{itemize}
\end{theorem}

\begin{figure}[h]
		\begin{center}
                \begin{tikzpicture}[main/.style = {fill = black}]
    
                    \draw[main] (1,0) circle (4pt);
                    \draw[main] (-1,0) circle (4pt);
                    \draw[main] (0,0) circle (4pt);
                    \draw[main] (0,1) circle (4pt);
                    \draw[main] (0,-1) circle (4pt);
    
                    \draw (1,0) -- (0,1)  -- (-1,0) -- (0,-1) -- (1,0);
                    \draw (-1,0) -- (0, 0) -- (1,0);
                    
                    \draw (-0.8,0.8) node{$e_1$};
                    \draw (0.8,0.8) node{$e_2$};
                    \draw (-0.5,0.2) node{$e_3$};
                    \draw (0.5,0.2) node{$e_4$};
                    \draw (-0.8,-0.8) node{$e_5$};
                    \draw (0.8,-0.8) node{$e_6$};

            \end{tikzpicture}
            \caption{The graph $\Theta_5$}
            \label{Fig:Theta5}
            \end{center}
\end{figure}
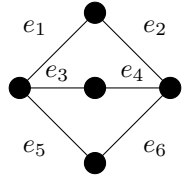

If either $\Gamma$ is bipartite or $(\Gamma,\gamma)$ is twist bipartite, then exactly half of the ordered pairs $(\vec{x},\sigma) \in S(m,n)$ are solvable, but for different reasons.  For example, the twist graph that represents the 15+4 puzzle, pictured in Figure~\ref{Fig:Twist15+4}, is twist bipartite.  Thus, by Theorem~\ref{Thm:MainThm}, the set of solvable ordered pairs $(\vec{x},\sigma) \in S(4,19)$ for the $15+4$ puzzle is
\[
\Big\{ (\vec{x},\sigma) \in S(4,19) \mbox{ } \vert \mbox{ } \sum_{i=1}^{19} x_i \equiv \mathrm{sign} (\sigma) \Mod{2} \Big\} .
\]
Thus, if you were to pop out one of the tiles, rotate it 90 degrees, and pop it back in again, the resulting puzzle would be unsolvable.  Similarly, if you were to pop out two of the tiles and swap them without rotating, the resulting puzzle would be unsolvable.  If, however, you were to pop out two of the tiles, swap them, and rotate one of them 90 degrees before popping them back in, the resulting puzzle would be solvable.

\subsection{Exceptional Cases}
If $\Gamma$ is a cycle, then the group of solvable permutations is cyclic.  In addition to the cycle graphs, Theorem~\ref{Thm:MainThm} has two exceptional cases.  Both of these are related to well-known exotic behaviors of the small-order symmetric and alternating groups.  Famously, $A_4$ is the only alternating group that is not simple.  The Klein 4-group $K$ is a normal subgroup of $A_4$, and we denote the quotient by
\[
q \colon A_4 \to A_4 /K \cong \Z/3\Z .
\]
Additionally, $S_6$ is the only symmetric group with a non-trivial outer automorphism.  This automorphism maps the ``standard'' copy of $S_5$ in $S_6$ isomorphically onto the ``exotic'' copy $\mathrm{PGL} (2,5) \cong S_5$.  The following characterizes the possible permutations and rotations that can occur in sliding block puzzles where the underlying graph is one of the two exceptional graphs $\Theta_5$ or $\Theta_7$.

\begin{theorem}
\label{Thm:Exceptional}
If $\Gamma = \Theta_5$ and $\varphi_{\gamma}$ is surjective, then the group of solvable order pairs $(\vec{x},\sigma) \in S(m,4)$ is:
\begin{enumerate}
\item $\Big\{ (\vec{x},\sigma) \in S(m,4) \mbox{ } \vert \mbox{ }\sigma \in A_4, \sum_{i=1}^4 x_i \equiv q(\sigma) \Mod{3} \Big\}$ \\ if $m$ is divisible by 3 and $\sum_{i=1}^6 \gamma_{e_i} \equiv 0 \Mod{3}$, and
\item $\Big\{ (\vec{x},\sigma) \in S(m,4) \mbox{ } \vert \mbox{ } \sigma \in A_4 \Big\}$ otherwise.
\end{enumerate}
If $\Gamma = \Theta_7$ and $\varphi_{\gamma}$ is surjective, then the group of solvable order pairs $(\vec{x},\sigma) \in S(m,6)$ is:
\begin{enumerate}
\item $\Big\{ (\vec{x},\sigma) \in S(m,6) \mbox{ } \vert \mbox{ }\sigma \in \mathrm{PGL}(2,5), \sum_{i=1}^6 x_i \equiv \mathrm{sign} (\sigma) \Mod{2} \Big\}$ \\ if $m$ is even and $(\Theta_7,\gamma)$ is twist bipartite, and
\item $\mathrm{PGL}(2,5)$ otherwise.
\end{enumerate}
\end{theorem}

\subsection{Other Puzzles}

The $15+4$ puzzle is not the only puzzle of this type.  Segerman's other creations include the hyperbolic 29-puzzle and the continental drift puzzle, pictured in Figures~\ref{Fig:Hyp29} and~\ref{Fig:Continental}, which he discusses in the YouTube videos:
\begin{align*}
\mbox{\url{https://www.youtube.com/watch?v=EitWHthBY30}} \\
\mbox{\url{https://www.youtube.com/watch?v=0uQx33KFMO0}.}
\end{align*}
The latter puzzle is a globe that is tiled like a soccer ball, where the 12 hexagonal faces are tiles that can slide around.  As in the $15+4$ puzzle, these are sliding block puzzles in which sliding the tiles may cause them to rotate.  Online versions of these and similar puzzles, written by the makers of HyperRogue, can be found at
\[
\mbox{\url{http://roguetemple.com/z/15/}.}
\]

\begin{figure}[h]
\begin{center}
\includegraphics[scale=0.2]{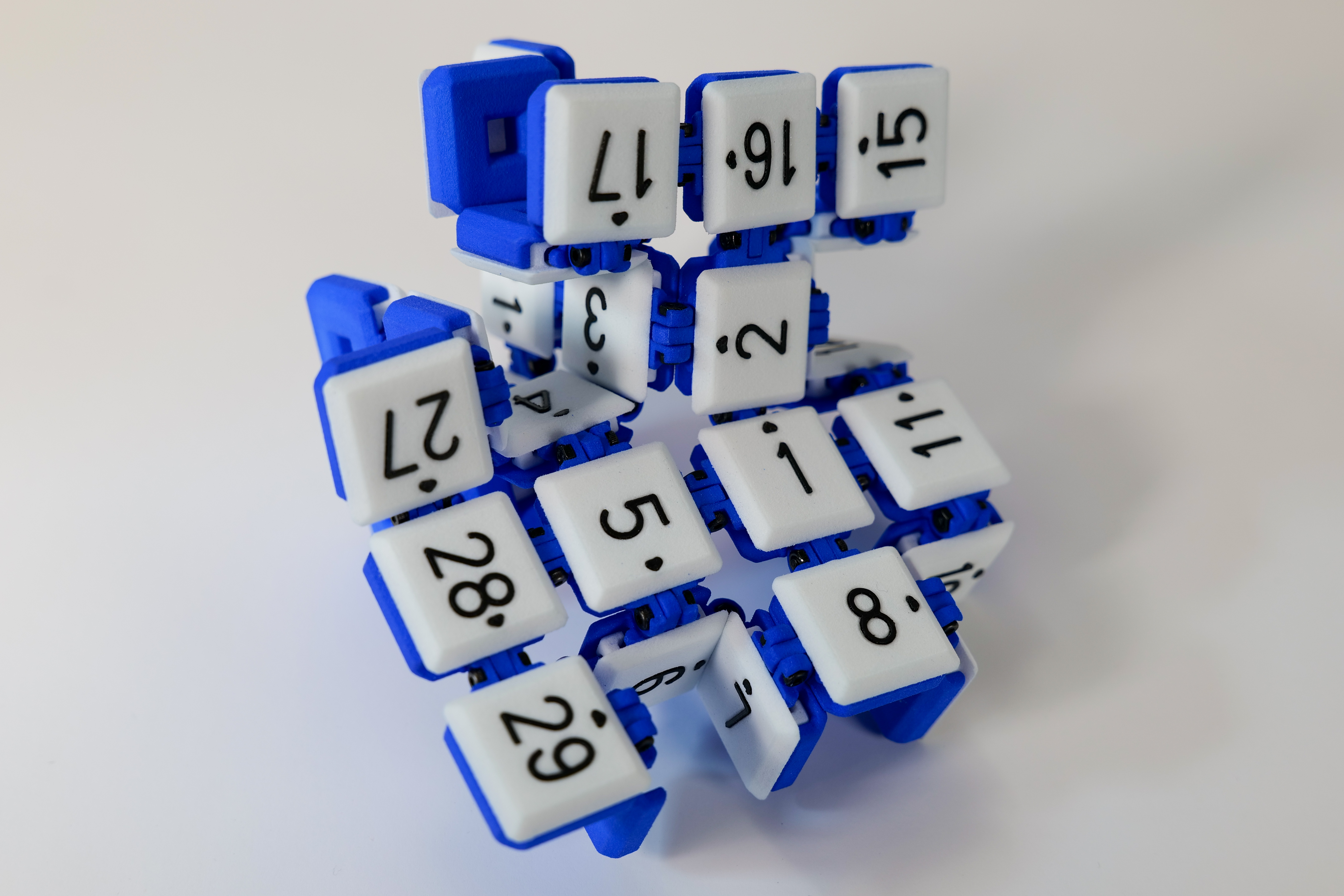}
\caption{Segerman's Hyperbolic 29-Puzzle (Photo courtesy of Henry Segerman)}
\label{Fig:Hyp29}
\end{center}
\end{figure}

\begin{figure}[h]
\begin{center}
\includegraphics[scale=0.15]{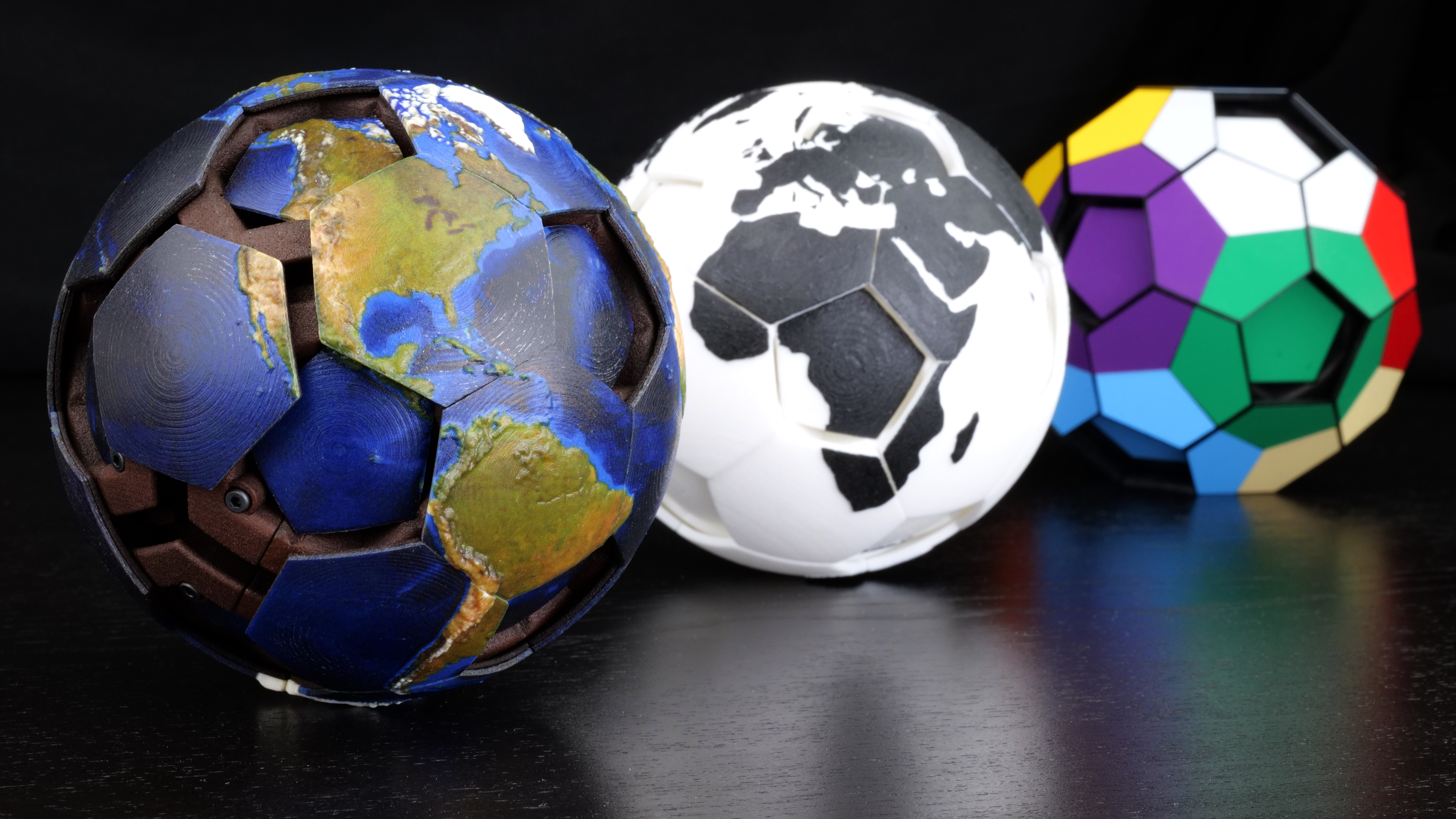}
\caption{Segerman's Continental Drift Puzzle (Photo courtesy of Henry Segerman)}
\label{Fig:Continental}
\end{center}
\end{figure}

\begin{remark}
\label{Rmk:Choice}
Given a puzzle like the $15+4$ puzzle or the continental drift puzzle, how do we construct the corresponding twist graph?  To start, we need to record how much a tile has been rotated even when that tile is not in its original position.  To do this, for each tile we choose one of its $m$ sides to call the ``top'' of the tile.  No matter where a tile is, it is unrotated if its top side lines up with the top side of the tile that is there in the puzzle's solved state.  In the $15+4$ puzzle, where the tiles have numbers written on them, there is a natural choice of top side for each tile.  If we consider the continental drift puzzle, however, it becomes clearer that this is an arbitrary choice.  For example, no side of the North Pole tile has a particularly strong claim to being its ``top''.  If we slide the tiles around so that the Madagascar tile is where the North Pole should be, none of the 6 possible rotations of the Madagascar tile is the natural ``correct'' one.  Thus, designating a top side for each tile forces us to make a choice.  The corresponding twist graph depends on how one makes this choice.
\end{remark}

The graphs representing the hyperbolic 29-puzzle and the continental drift puzzle are pictured in Figure~\ref{Fig:OtherPuzzles}.  In the first case, $m=4$ and in the second, $m=6$.  Both graphs contain a simple closed path of length 5, so neither is bipartite.  Following the remark above, one can choose the ``top'' side of each tile so that $\gamma_e \equiv \pm 1 \Mod{m}$ for every oriented edge $e$ in $\Gamma$.  Specifically, in Figure~\ref{Fig:OtherPuzzles}, the edges are labeled with arrows so that $\gamma_e \equiv 1 \Mod{m}$ when $e$ is oriented in the direction of the arrow.  Since $\gamma_e$ is odd for all $e$, both twist graphs are twist bipartite.  Finally, both graphs are planar, and if $\omega \in H_1 (\Gamma, \Z/m\Z)$ is a face of the planar embedding, then $\varphi_{\gamma} (\omega) \equiv \pm 1 \Mod{m}$.  (When $m=4$, this follows immediately from the fact that $\gamma_e \equiv \pm 1 \Mod{m}$ for all oriented edges $e$.  When $m=6$, it suffices to check that no pentagonal face has exactly 4 edges oriented either clockwise or counter-clockwise.)  Therefore, $\varphi_{\gamma}$ is surjective.  Thus, by Theorem~\ref{Thm:MainThm}, the set of solvable order pairs for the hyperbolic 29-puzzle is
\[
\Big\{ (\vec{x},\sigma) \in S(4,29) \mbox{ } \vert \mbox{ } \sum_{i=1}^{29} x_i \equiv \mathrm{sign} (\sigma) \Mod{2} \Big\} ,
\]
and the set of solvable ordered pairs for the continental drift puzzle is
\[
\Big\{ (\vec{x},\sigma) \in S(6,11) \mbox{ } \vert \mbox{ } \sum_{i=1}^{11} x_i \equiv \mathrm{sign} (\sigma) \Mod{2} \Big\} .
\]

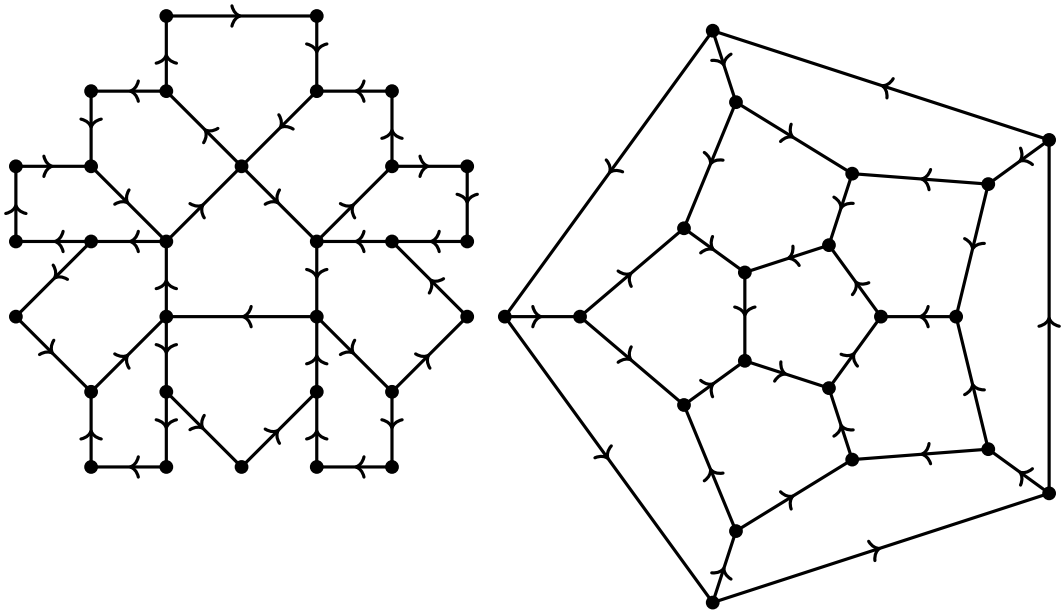
\begin{figure}[h]
		\begin{center}
               \begin{tikzpicture}[very thick,decoration={
    markings,
    mark=at position 0.5 with {\arrow{>}}}
    ]
\coordinate (A1) at ({7.5+cos(0)},{sin(0)});
\coordinate (A2) at ({7.5+cos(360/5)},{sin(360/5)});
\coordinate (A3) at ({7.5+cos(2*360/5)},{sin(2*360/5)});
\coordinate (A4) at ({7.5+cos(3*360/5)},{sin(3*360/5)});
\coordinate (A5) at ({7.5+cos(4*360/5)},{sin(4*360/5)});

\filldraw[black] (A1) circle (2pt);
\filldraw[black] (A2) circle (2pt);
\filldraw[black] (A3) circle (2pt);
\filldraw[black] (A4) circle (2pt);
\filldraw[black] (A5) circle (2pt);

\coordinate (B1) at ({7.5+2*cos(0)},{2*sin(0)});
\coordinate (B2) at ({7.5+2*cos(360/5)},{2*sin(360/5)});
\coordinate (B3) at ({7.5+2*cos(2*360/5)},{2*sin(2*360/5)});
\coordinate (B4) at ({7.5+2*cos(3*360/5)},{2*sin(3*360/5)});
\coordinate (B5) at ({7.5+2*cos(4*360/5)},{2*sin(4*360/5)});

\filldraw[black] (B1) circle (2pt);
\filldraw[black] (B2) circle (2pt);
\filldraw[black] (B3) circle (2pt);
\filldraw[black] (B4) circle (2pt);
\filldraw[black] (B5) circle (2pt);

\coordinate (C1) at ({7.5+3*cos(0+360/10)},{3*sin(0+360/10)});
\coordinate (C2) at ({7.5+3*cos(360/5+360/10)},{3*sin(360/5+360/10)});
\coordinate (C3) at ({7.5+3*cos(2*360/5+360/10)},{3*sin(2*360/5+360/10)});
\coordinate (C4) at ({7.5+3*cos(3*360/5+360/10)},{3*sin(3*360/5+360/10)});
\coordinate (C5) at ({7.5+3*cos(4*360/5+360/10)},{3*sin(4*360/5+360/10)});

\filldraw[black] (C1) circle (2pt);
\filldraw[black] (C2) circle (2pt);
\filldraw[black] (C3) circle (2pt);
\filldraw[black] (C4) circle (2pt);
\filldraw[black] (C5) circle (2pt);

\coordinate (D1) at ({7.5+4*cos(0+360/10)},{4*sin(0+360/10)});
\coordinate (D2) at ({7.5+4*cos(360/5+360/10)},{4*sin(360/5+360/10)});
\coordinate (D3) at ({7.5+4*cos(2*360/5+360/10)},{4*sin(2*360/5+360/10)});
\coordinate (D4) at ({7.5+4*cos(3*360/5+360/10)},{4*sin(3*360/5+360/10)});
\coordinate (D5) at ({7.5+4*cos(4*360/5+360/10)},{4*sin(4*360/5+360/10)});

\filldraw[black] (D1) circle (2pt);
\filldraw[black] (D2) circle (2pt);
\filldraw[black] (D3) circle (2pt);
\filldraw[black] (D4) circle (2pt);
\filldraw[black] (D5) circle (2pt);

\draw[postaction={decorate}] (A1) -- (A2);
\draw[postaction={decorate}] (A2) -- (A3);
\draw[postaction={decorate}] (A3) -- (A4);
\draw[postaction={decorate}] (A4) -- (A5);
\draw[postaction={decorate}] (A5) -- (A1);

\draw[postaction={decorate}] (C1) -- (B1);
\draw[postaction={decorate}] (C2) -- (B2);
\draw[postaction={decorate}] (C3) -- (B3);
\draw[postaction={decorate}] (C4) -- (B4);
\draw[postaction={decorate}] (C5) -- (B5);

\draw[postaction={decorate}] (B1) -- (A1);
\draw[postaction={decorate}] (B2) -- (A2);
\draw[postaction={decorate}] (B3) -- (A3);
\draw[postaction={decorate}] (B4) -- (A4);
\draw[postaction={decorate}] (B5) -- (A5);

\draw[postaction={decorate}] (C5) -- (B1);
\draw[postaction={decorate}] (C1) -- (B2);
\draw[postaction={decorate}] (C2) -- (B3);
\draw[postaction={decorate}] (C3) -- (B4);
\draw[postaction={decorate}] (C4) -- (B5);

\draw[postaction={decorate}] (D1) -- (C1);
\draw[postaction={decorate}] (D2) -- (C2);
\draw[postaction={decorate}] (D3) -- (C3);
\draw[postaction={decorate}] (D4) -- (C4);
\draw[postaction={decorate}] (D5) -- (C5);

\draw[postaction={decorate}] (D1) -- (D2);
\draw[postaction={decorate}] (D2) -- (D3);
\draw[postaction={decorate}] (D3) -- (D4);
\draw[postaction={decorate}] (D4) -- (D5);
\draw[postaction={decorate}] (D5) -- (D1);

\coordinate (B) at (-1,1);
\coordinate (C) at (0,2);
\coordinate (D) at (1,1);
\coordinate (E) at (1,0);
\coordinate (A) at (-1,0);
\coordinate (F) at (-2,1);
\coordinate (G) at (-3,0);
\coordinate (H) at (-2,-1);
\coordinate (I) at (-2,-2);

\coordinate (J) at (-1,-2);
\coordinate (K) at (-1,-1);
\coordinate (L) at (0,-2);
\coordinate (M) at (1,-1);
\coordinate (N) at (1,-2);
\coordinate (O) at (2,-2);
\coordinate (P) at (2,-1);
\coordinate (Q) at (3,0);
\coordinate (R) at (2,1);
\coordinate (S) at (3,1);
\coordinate (T) at (3,2);
\coordinate (U) at (2,2);
\coordinate (V) at (2,3);
\coordinate (W) at (1,3);
\coordinate (Z) at (1,4);
\coordinate (A1) at (-1,4);
\coordinate (B1) at (-1,3);
\coordinate (C1) at (-2,3);
\coordinate (D1) at (-2,2);
\coordinate (E1) at (-3,2);
\coordinate (F1) at (-3,1);

\filldraw[black] (A) circle (2pt);
\filldraw[black] (B) circle (2pt);
\filldraw[black] (C) circle (2pt);
\filldraw[black] (D) circle (2pt);
\filldraw[black] (E) circle (2pt);
\filldraw[black] (F) circle (2pt);
\filldraw[black] (G) circle (2pt);
\filldraw[black] (H) circle (2pt);
\filldraw[black] (I) circle (2pt);
\filldraw[black] (J) circle (2pt);
\filldraw[black] (K) circle (2pt);
\filldraw[black] (L) circle (2pt);
\filldraw[black] (M) circle (2pt);
\filldraw[black] (N) circle (2pt);
\filldraw[black] (O) circle (2pt);
\filldraw[black] (P) circle (2pt);
\filldraw[black] (Q) circle (2pt);
\filldraw[black] (R) circle (2pt);
\filldraw[black] (S) circle (2pt);
\filldraw[black] (T) circle (2pt);
\filldraw[black] (U) circle (2pt);
\filldraw[black] (V) circle (2pt);
\filldraw[black] (W) circle (2pt);
\filldraw[black] (Z) circle (2pt);
\filldraw[black] (A1) circle (2pt);
\filldraw[black] (B1) circle (2pt);
\filldraw[black] (C1) circle (2pt);
\filldraw[black] (D1) circle (2pt);
\filldraw[black] (E1) circle (2pt);
\filldraw[black] (F1) circle (2pt);

\draw[postaction={decorate}] (A) -- (B);
\draw[postaction={decorate}] (B) -- (C);
\draw[postaction={decorate}] (C) -- (D);
\draw[postaction={decorate}] (D) -- (E);
\draw[postaction={decorate}] (E) -- (A);
\draw[postaction={decorate}] (B) -- (F);
\draw[postaction={decorate}] (F) -- (G);
\draw[postaction={decorate}] (G) -- (H);
\draw[postaction={decorate}] (H) -- (A);
\draw[postaction={decorate}] (I) -- (H);
\draw[postaction={decorate}] (J) -- (I);
\draw[postaction={decorate}] (K) -- (J);
\draw[postaction={decorate}] (A) -- (K);
\draw[postaction={decorate}] (K) -- (L);
\draw[postaction={decorate}] (L) -- (M);
\draw[postaction={decorate}] (M) -- (E);
\draw[postaction={decorate}] (N) -- (M);
\draw[postaction={decorate}] (O) -- (N);
\draw[postaction={decorate}] (P) -- (O);
\draw[postaction={decorate}] (E) -- (P);
\draw[postaction={decorate}] (P) -- (Q);
\draw[postaction={decorate}] (Q) -- (R);
\draw[postaction={decorate}] (R) -- (D);
\draw[postaction={decorate}] (S) -- (R);
\draw[postaction={decorate}] (T) -- (S);
\draw[postaction={decorate}] (U) -- (T);
\draw[postaction={decorate}] (D) -- (U);
\draw[postaction={decorate}] (U) -- (V);
\draw[postaction={decorate}] (V) -- (W);
\draw[postaction={decorate}] (W) -- (C);
\draw[postaction={decorate}] (Z) -- (W);
\draw[postaction={decorate}] (A1) -- (Z);
\draw[postaction={decorate}] (B1) -- (A1);
\draw[postaction={decorate}] (C) -- (B1);
\draw[postaction={decorate}] (B1) -- (C1);
\draw[postaction={decorate}] (C1) -- (D1);
\draw[postaction={decorate}] (D1) -- (B);
\draw[postaction={decorate}] (F) -- (F1);
\draw[postaction={decorate}] (F1) -- (E1);
\draw[postaction={decorate}] (E1) -- (D1);
\end{tikzpicture}
            \caption{Graph representations of the hyperbolic 29-puzzle (left) and the continental drift puzzle (right)}
            \label{Fig:OtherPuzzles}
            \end{center}
\end{figure}

In the puzzles we study here, each tile has a cyclic group of rotations, but one can imagine other possibilities.  Figure~\ref{Fig:Cube} depicts a sliding block puzzle whose ``tiles'' are cubes.  The group of rotations of each individual tile is the rotation group of the cube, which is isomorphic to $S_4$.  One potential line of future inquiry could be to study puzzles of this type, where the group of rotations of the tiles are not cyclic.

\begin{figure}[h]
\begin{center}
\includegraphics[scale=0.15]{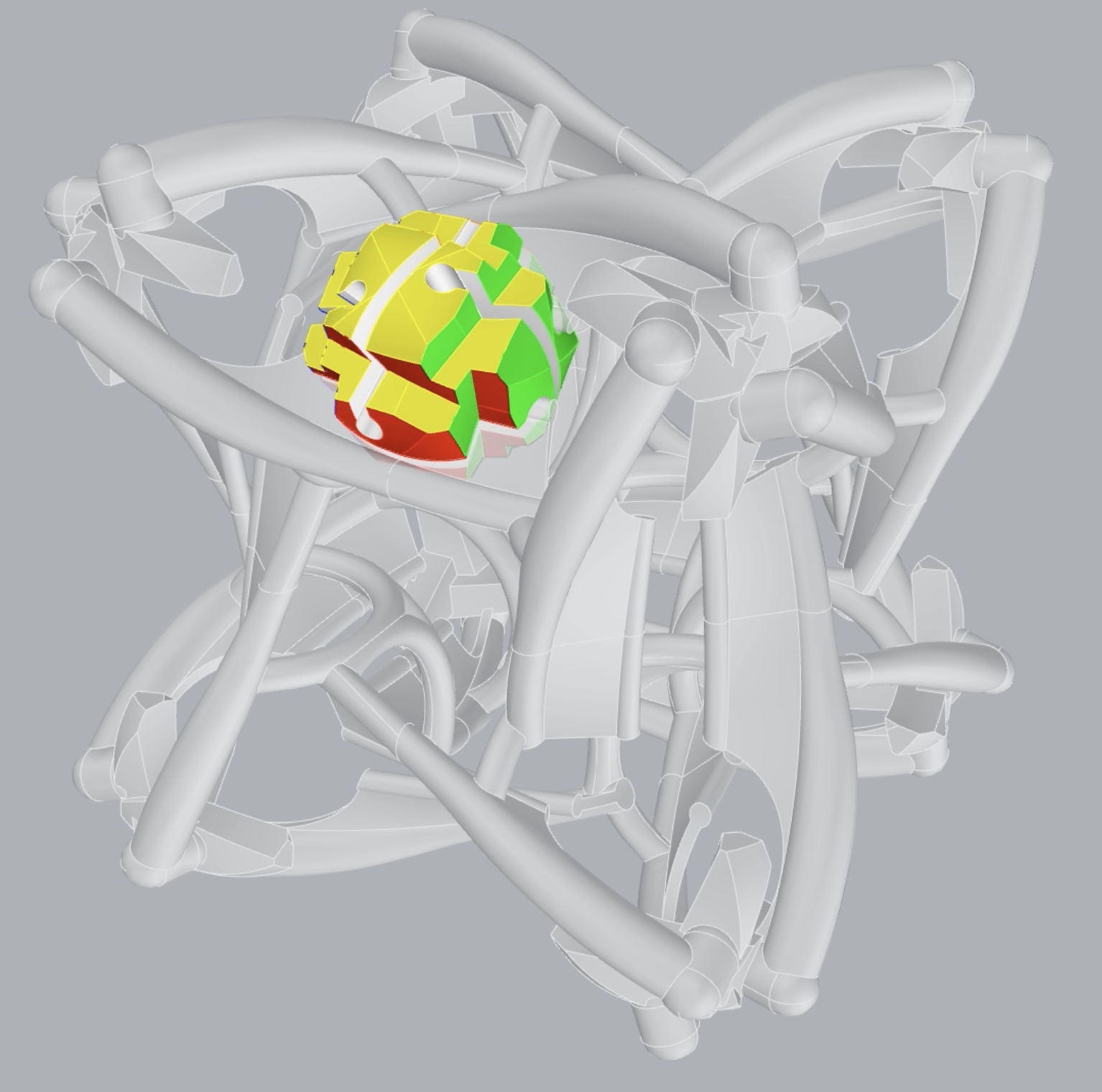}
\caption{A sliding block puzzle whose tiles are cubes (Photo courtesy of Henry Segerman)}
\label{Fig:Cube}
\end{center}
\end{figure}

\subsection{Outline of the paper}
In Section~\ref{Sec:FundamentalGroup}, we describe Wilson's construction of the group of solvable permutations, and some basic properties of the fundamental groups of graphs.  In Section~\ref{Sec:GenSym}, we discuss the generalized symmetric groups $S(m,n)$, and see that the set of solvable order pairs is a subgroup of $S(m,n)$.  Finally, in Section~\ref{Sec:Proof}, we prove Theorems~\ref{Thm:MainThm} and~\ref{Thm:Exceptional}.

\section{The fundamental group}
\label{Sec:FundamentalGroup}

In this section, we briefly describe the fundamental group of a graph and Wilson's construction of a homomorphism from this group to the symmetric group.  The fundamental group can be defined for more general (pointed) topological spaces (see, for example \cite[Proposition~1.3]{Hatcher}), but for our purposes, we will only need the fundamental groups of graphs.

\subsection{Paths and Homotopy}
A \emph{path} $p$ in a graph $\Gamma$ is a sequence of oriented edges
\[
p = e_1 e_2 \cdots e_k
\]
such that the head of $e_i$ is equal to the tail of $e_{i+1}$ for all $i$.  The path $p$ is called a \emph{closed path} if  the head of $e_k$ is equal to the tail of $e_1$.  In other words, a closed path is a path that starts and ends at the same vertex.  The path $p$ is called \emph{simple} if the heads of the edges $e_i$ are distinct and the tails of $e_i$ are distinct.  That is, a path is simple if it does not pass through a vertex more than once, with the possible exception that it may start and end at the same vertex.

If two closed paths $p = e_1 \dots e_k, q = f_1 \cdots f_{\ell}$ both start and end at the same vertex $v$, then their concatenation $pq = e_1 \cdots e_k f_1 \cdots f_{\ell}$ is a closed path that starts and ends at $v$.  This defines a product on the set of closed paths that start and end at the given vertex $v$.  This does not, however, define a group structure on this set, since there are no inverses.  To fix this, we define an equivalence relation on this set.

The closed path $p = e\overline{e}$ consisting of an oriented edge, followed by that same edge with the opposite orientation, is called the \emph{irrelevant closed path}.  We say that two paths are \emph{homotopy equivalent} if one can be obtained from the other by a sequence of inserting and deleting irrelevant closed paths.  The \emph{fundamental group} $\pi_1 (\Gamma , v)$ is the set of homotopy equivalence classes of closed paths that start and end at the vertex $v$, under the operation of concatenation.  The identity element of this group is the empty path, and the inverse of a path $p = e_1 \cdots e_k$ is the \emph{reverse path} $\overline{p} = \overline{e}_k \cdots \overline{e}_1$.

The first homology group $H_1 (\Gamma,\Z)$ defined in Section~\ref{Sec:Intro} is the abelianization of $\pi_1 (\Gamma,v)$.  In particular, any homomorphism from $\pi_1 (\Gamma,v)$ to an abelian group factors through $H_1 (\Gamma,\Z)$.  Moreover, any homomorphism from $\pi_1 (\Gamma,v)$ to an $m$-torsion abelian group factors through $H_1 (\Gamma,\Z/m\Z)$.

\subsection{Wilson's Construction}
Given an oriented edge $e$ in a graph $\Gamma$, Wilson defines $\sigma_e$ to be the permutation of $V(\Gamma)$ that transposes the head and tail of $e$.  For a path $p = e_1 \cdots e_k$, he then defines the permutation $\sigma_p$ by composing\footnote{In his paper, Wilson uses the convention that permutations are composed from left to right, whereas we prefer to compose them from right to left.} the transpositions
\[
\sigma_p = \sigma_{e_k} \circ \cdots \circ \sigma_{e_2} \circ \sigma_{e_1}.
\]
If $\Gamma$ represents a sliding block puzzle and the path $p$ starts at the empty vertex, then $\sigma_p$ is the permutation of the tiles obtained by sliding the tiles along the path $p$.  In other words, a tile that is initially located at a vertex $w$ will move to $\sigma_p (w)$ after performing this sequence of moves.  Note that, if $p$ is a simple closed path of length $k$, then $\sigma_p$ is the $(k-1)$-cycle
\[
\sigma_p = \Big( h(e_1) h(e_2) \cdots h(e_{k-1}) \Big),
\]
where $h(e_i)$ denotes the head of $e_i$.

Because every transposition is its own inverse, the map $\sigma$ takes homotopy equivalent paths to the same permutation.  Thus, it descends to a well-defined map on the fundamental group $\pi_1 (\Gamma , v)$.  By construction, this map is a homomorphism, and every permutation in the image fixes the vertex $v$.  The group of solvable permutations in which the empty vertex is fixed is simply the image of the homomorphism $\sigma$.  The possibilities for this image are classified in Theorem~\ref{Thm:Wilson}.

\begin{remark}
The set of solvable permutations that do \emph{not} fix the empty vertex does not form a group, because one can only concatenate two permutations if the position of the empty vertex at the end of the first permutation agrees with that at the start of the second permutation.  This set does form an arguably more natural object called a \emph{groupoid}, but this is beyond the scope of the present article \cite[Definition~1.1.11]{Riehl}.
\end{remark}

\subsection{Generators and the Map $\varphi_{\gamma}$}

The fundamental group of a graph has a well-known set of generators.

\begin{lemma}
\label{Lem:Generators}
Let $\Gamma$ be a connected graph, and let $v$ be any vertex in $\Gamma$.  Then $\pi_1 (\Gamma, v)$ is freely generated by $g$ simple closed paths, where $g = \vert E(\Gamma) \vert - \vert V(\Gamma) \vert + 1$.  Moreover if $\Gamma$ is not a cycle, or obtained from a cycle by replacing an edge with multiple parallel edges, then these simple closed paths can be chosen so they are non-Hamiltonian.  
\end{lemma}

\begin{proof}
The first part is standard.  Let $T \subseteq \Gamma$ be a spanning tree, and let $e_1 , \ldots , e_g$ be the edges of $\Gamma$ that are not contained in $T$.  For each $e_i$, there is a unique simple path in $T$ from the head of $e_i$ to its tail.  Appending $e_i$ to the end of this path, we obtain a simple closed path $p_i$ in $\Gamma$.  By \cite[Proposition~1A.2]{Hatcher}, the fundamental group $\pi_1 (\Gamma)$ is the free group on the generators $p_1 , \ldots , p_g$.

For the second part, if $\Gamma$ is not a cycle, or obtained from a cycle by replacing an edge with multiple parallel edges, then either $\Gamma$ is a tree, in which case the fundamental group is trivial, or there exists a vertex $w$ in $\Gamma$ that has at least 3 neighbors.  Construct a spanning tree $T$ by breadth-first search starting from $w$.  By construction, the simple closed paths $p_i$ contain at most 2 of the neighbors of $w$.  Hence, they are non-Hamiltonian.
\end{proof}

Lemma~\ref{Lem:Generators} allows us to describe a large family of twist graphs $(\Gamma,\gamma)$ for which $\varphi_{\gamma}$ is surjective.  Specifically, consider the subgraph $\Gamma_0 \subseteq \Gamma$ with $V(\Gamma_0) = V(\Gamma)$ and $E(\Gamma_0) = \{ e \in E(\Gamma) \vert \gamma_e \equiv 0 \Mod{m} \}$.  In other words, $\Gamma_0$ is the subgraph of $\Gamma$ consisting of edges that do not rotate the tiles.  We first note the following.

\begin{lemma}
\label{Lem:PuzzleToGraph}
Every sliding block puzzle can be represented by a twist graph $(\Gamma,\gamma)$ such that $\Gamma_0$ is connected.
\end{lemma}

\begin{proof}
The graph $\Gamma$ is uniquely determined by the puzzle -- it has a vertex for every possible position of the tiles and an edge between two vertices if the corresponding vertices are adjacent.  Let $T \subseteq \Gamma$ be a rooted spanning tree.  Now, arbitrarily choose a side of the tile at the root vertex to be the ``top'' of the tile.  We then inductively choose the top side of each tile so that $\gamma_e = 0$ for all edges $e$ in $T$.  Specifically, let $v$ be a vertex for which we have already chosen a top side, let $w$ be connected to $v$ by an edge $e$ in $T$, and suppose that we have not already chosen a top side for $w$.  We can then choose the top side of the tile at $w$ so that $\gamma_e = 0$.  Proceeding in this way from the root vertex to each vertex in $\Gamma$, we see that $T \subseteq \Gamma_0$, so $\Gamma_0$ is connected.
\end{proof}

Lemma~\ref{Lem:PuzzleToGraph} is important for the following reason.

\begin{lemma}
\label{Lem:Gamma0Connected}
Suppose that $\Gamma_0$ is connected.  The map $\varphi_{\gamma}$ is surjective if and only if the set $\{ \gamma_e \vert e \in E(\Gamma) \}$ generates $\Z/m\Z$.  Equivalently, $\varphi_{\gamma}$ is surjective if and only if the set $\{ m, \gamma_e \vert e \in E(\Gamma) \}$ has no common divisor greater than 1.
\end{lemma}

\begin{proof}
Since $\Gamma_0$ is connected, it contains a spanning tree $T$.  The subgraph $\Gamma_0$ has the same set of vertices as $\Gamma$, so $T$ is also a spanning tree for $\Gamma$.  In the construction of Lemma~\ref{Lem:Generators}, each of the simple closed paths $p_i$ contains exactly 1 edge that is not in the spanning tree.  Thus, each of the simple closed paths $p_i$ contains at most 1 edge $e_i$ such that $\gamma_{e_i} \not\equiv 0 \Mod{m}$.  Because these simple closed paths generate $\pi_1 (\Gamma,v)$, their images in first homology generate $H_1 (\Gamma,\Z)$.  It follows that the map $\varphi_{\gamma}$ is surjective if and only if the elements $\gamma_{e_1} , \ldots \gamma_{e_g}$ generate $\Z/m\Z$.  The last statement holds because the subgroups of $\Z/m\Z$ are precisely the cyclic subgroups generated by divisors of $m$.
\end{proof}

For example, consider the twist graph $(\Gamma,\gamma)$ pictured in Figure~\ref{Fig:Twist15+4}, representing the $15+4$ puzzle.  The subgraph $\Gamma_0$ is obtained by deleting the two dashed edges.  We see that $\Gamma_0$ is connected.  If $e$ is one of the dashed edges, oriented from left to right, then $\gamma_e \equiv 1 \Mod{4}$, which generates $\Z/4\Z$.  Thus, $\varphi_{\gamma}$ is surjective.

On the other hand, if there exists an integer $a>1$ dividing both $m$ and $\gamma_e$ for all $e \in E(\Gamma)$, then the only possible rotations are multiples of $a$.  It is therefore natural to recast the problem -- instead of the tiles having $m$ sides, one should think of them as having $\frac{m}{a}$ sides, and we obtain a twist graph with coefficients in $\Z/\frac{m}{a}\Z$ by dividing $\gamma_e$ by $a$ for all $e$.  Reducing the problem in this way, one can apply Theorem~\ref{Thm:MainThm} to all twist graphs for which $\Gamma_0$ is connected.

\section{Generalized symmetric groups}
\label{Sec:GenSym}

\subsection{Examining the group structure}
In this section, we generalize Wilson's construction to record both the permutation of the tiles and their possible rotations.  Given a path $p$ in the graph $\Gamma$, we will define an ordered pair $(\vec{x}(p),\sigma_p)$, where $\sigma_p$ is the permutation of the previous section, and $\vec{x}(p)$ is a vector with coefficients in $\Z/m\Z$.  First, however, we need to define the group structure on this set of ordered pairs.  Importantly, we will see that the product group $(\Z/m\Z)^n \times S_n$ is not appropriate for the task, because in sliding block puzzles the permutations and rotations do not commute.

To see this, consider the sliding block puzzle pictured on the left in Figure~\ref{Fig:NoCommute}.  In this example, $m=2$, and sliding a tile across the dashed edge rotates that tile by 180 degrees.  The empty vertex is the upper left one.  The closed path $p$ starts at the empty vertex, proceeds to the right along the solid edge, and then back to the left along the dashed edge.   Sliding the tiles along the closed path $p$ does not permute the tiles, but does rotate the tile labeled 1.  The closed path $q$ starts at the empty vertex and proceeds counter-clockwise around the solid triangle.  Sliding the tiles along the closed path $q$ transposes the tiles labeled 1 and 2, but does not rotate any tiles.  On the righthand side of the figure, we see how the tiles are rotated and permuted by the closed paths $pq$ (on the top) and $qp$ (on the bottom).  Crucially, we see that they are not the same.

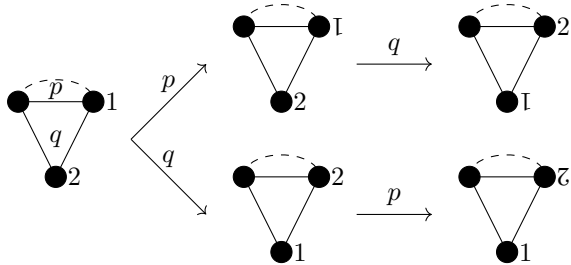
\begin{figure}[h]
		\begin{center}
                \begin{tikzpicture}[main/.style = {fill = black}]
    
                    \draw[main] (0,0) circle (4pt);
                    \draw[main] (1,0) circle (4pt);
                    \draw[main] (0.5,-1) circle (4pt);
                    \draw (0,0) -- (1,0)  -- (0.5,-1) -- (0,0);
                    \draw [dashed] (0,0) to [out=90, in=90] (1,0);                     
                    \draw (0.5,-0.5) node{$q$};
                    \draw (0.5,0.15) node{$p$};
                    \draw (1.25,0) node{$1$};
                    \draw (0.75,-1) node{$2$};
                    
                    \draw[->] (1.5,-0.5)--(2.5,0.5) node[midway, above] {$p$};
                    
                    \draw[main] (3,1) circle (4pt);
                    \draw[main] (4,1) circle (4pt);
                    \draw[main] (3.5,0) circle (4pt);    
                    \draw (3,1) -- (4,1)  -- (3.5,0) -- (3,1);
                    \draw [dashed] (3,1) to [out=90, in=90] (4,1);                     
                    \draw (4.25,1)  node[rotate=180]{$1$};
                    \draw (3.75,0) node{$2$};
                    
                    \draw[->] (1.5,-0.5)--(2.5,-1.5) node[midway, above] {$q$};
                    
                    \draw[main] (3,-1) circle (4pt);
                    \draw[main] (4,-1) circle (4pt);
                    \draw[main] (3.5,-2) circle (4pt);
                    \draw (3,-1) -- (4,-1)  -- (3.5,-2) -- (3,-1);
                    \draw [dashed] (3,-1) to [out=90, in=90] (4,-1);                     
                    \draw (4.25,-1) node{$2$};
                    \draw (3.75,-2) node{$1$};
                    
                    \draw[->] (4.5,0.5)--(5.5,0.5) node[midway, above] {$q$};
                    
                    \draw[main] (6,1) circle (4pt);
                    \draw[main] (7,1) circle (4pt);
                    \draw[main] (6.5,0) circle (4pt);    
                    \draw (6,1) -- (7,1)  -- (6.5,0) -- (6,1);
                    \draw [dashed] (6,1) to [out=90, in=90] (7,1);                     
                    \draw (7.25,1)  node{$2$};
                    \draw (6.75,0) node[rotate=180]{$1$};
                    
                    \draw[->] (4.5,-1.5)--(5.5,-1.5) node[midway, above] {$p$};
                    
                    \draw[main] (6,-1) circle (4pt);
                    \draw[main] (7,-1) circle (4pt);
                    \draw[main] (6.5,-2) circle (4pt);
                    \draw (6,-1) -- (7,-1)  -- (6.5,-2) -- (6,-1);
                    \draw [dashed] (6,-1) to [out=90, in=90] (7,-1);                     
                    \draw (7.25,-1) node[rotate=180]{$2$};
                    \draw (6.75,-2) node{$1$};

            \end{tikzpicture}
            \caption{Permutations and rotations do not commute}
            \label{Fig:NoCommute}
            \end{center}
\end{figure}

To understand the group structure on this set of ordered pairs, first note that the map $\sigma$ of the previous section should factor through this group.  Thus, the kernel $(\Z/m\Z)^n$, consisting of ordered pairs of the form $(\vec{x}, \mathrm{id})$, is a normal subgroup.  By \cite[Theorem~5.12]{DF}, the group is a semidirect product $(\Z/m\Z)^n \rtimes S_n$.  To determine the full group structure, we need to consider the action by conjugation of $S_n$ on this normal subgroup.  To that end, and generalizing the previous example, suppose we have two closed paths $p$ and $q$ in an arbitrary sliding block puzzle.  Suppose further that sliding the tiles along $p$ rotates some of the tiles but does not permute them, and conversely, sliding tiles along $q$ permutes the tiles but does not rotate them.  Consider what happens when we slide the tiles along the composite path $qp\overline{q}$.  We see that the tile located at a vertex $v$ first slides to vertex $\sigma_q (v)$, where it is then rotated by $\vec{x}(p)_{\sigma_q (v)}$, and then moved back to the vertex $v$.  Thus, 
\[
\vec{x}(qp\overline{q})_v = \sigma_q^{-1} \cdot \vec{x}(p) \cdot \sigma_q = \vec{x}(p)_{\sigma_q (v)}.
\]
In other words, $S_n$ acts on $(\Z/m\Z)^n$ by reindexing.

\subsection{Generalized symmetric groups}
This action completely determines the group structure on our semidirect product -- it is the wreath product $(\Z/m\Z ) \wr S_n$, commonly known as the \emph{generalized symmetric group} $S(m,n)$.  The generalized symmetric groups have been studied extensively, dating back at least to Coxeter in 1936 \cite[Section~5]{Coxeter}.  The generalized symmetric group $S(m,n)$ is isomorphic to the group of generalized permutation matrices (that is, matrices in which each row and each column has exactly one nonzero entry), where the nonzero entries are $m$th roots of unity.

The generalized symmetric group has several natural quotients, which we record here.  First, it surjects onto the symmetric group:
\begin{align*}
\pi : S(m,n) \to S_n & \hspace{1 in} \pi (\vec{x},\sigma) = \sigma .
\end{align*}
For any integer $a$ dividing $m$, we have the maps:
\begin{align*}
\rho_{m,a} : S(m,n) \to S(a,n)& \hspace{1 in} \rho_{m,a} (\vec{x},\sigma) = (\vec{x} \Mod{a}, \sigma), \\
\eta_{m,a} : S(m,n) \to \Z/a\Z & \hspace{1 in} \eta_{m,a} (\vec{x},\sigma) = \sum_{i=1}^n x_i \Mod{a} .
\end{align*}
These maps fit into the following commutative diagram.
\[
\xymatrix{
S(m,n) \ar[d]|-{\eta_{m,m}} \ar[rrdd]|-{\eta_{m,a}} \ar[rrd]|-{\rho_{m,a}} \ar@/^1pc/[rrrrd]|-{\pi} & & & & \\
\Z/m\Z \ar[rrd] & & S(a,n) \ar[d]|-{\eta_{a,a}} \ar[rr]|-{\pi} & & S_n \\
& & \Z/a\Z & &  }
\]

In the special case where $m=2$, the generalized symmetric group $S(2,n)$ is also known as the \emph{signed permutation group} or the \emph{hyperoctahedral group}, and is isomorphic to the Coxeter group of type $B_n = C_n$.  The subgroup appearing in the introduction,
\[
\Big\{ (\vec{x},\sigma) \in S(2,n) \mbox{ } \vert \mbox{ } \eta_{2,2} (\vec{x},\sigma) \equiv \mathrm{sign} (\sigma) \Mod{2} \Big\} ,
\]
is isomorphic to the Coxeter group of type $D_n$.

\subsection{Generalizing Wilson's construction}
Returning to our construction, given an oriented edge $e$ in a twist graph $(\Gamma, \gamma)$, we define 
\[
\vec{x}(e) \in \prod_{v \in V(\Gamma)} \Z/m\Z
\]
to be the vector
\begin{displaymath}
\vec{x}(e)_v = \left\{ \begin{array}{ll}
\gamma_e & \textrm{if $v$ is the tail of $e$,} \\
0 & \textrm{otherwise.}
\end{array} \right.
\end{displaymath}
As in Wilson's construction, for a path $p = e_1 \cdots e_k$, we define the ordered pair $(\vec{x}(p),\sigma_p)$ by multiplying the ordered pairs 
\[
(\vec{x}(p),\sigma_p) := (\vec{x}(e_k),\sigma_{e_k}) \cdots (\vec{x}(e_2),\sigma_{e_2}) \cdot (\vec{x}(e_1),\sigma_{e_1})
\]
in the generalized symmetric group $S(m,n)$.  If $(\Gamma,\gamma)$ represents a sliding block puzzle and the path $p$ starts at the empty vertex, then $\vec{x}(p)$ records the rotations of the tiles obtained by sliding the tiles along the path $p$.

Because $\Z/m\Z$ is an $m$-torsion abelian group, the composition of the map from $\pi_1 (\Gamma,v)$ to $S(m,n)$ with the map $\eta_{m,m} : S(m,n) \to \Z/m\Z$ factors through $H_1 (\Gamma,\Z/m\Z)$.  The factorization $H_1 (\Gamma,\Z/m\Z) \to \Z/m\Z$ is the map $\varphi_{\gamma}$ from the introduction.  In other words, the following diagram commutes.
\[
\xymatrix{
\pi_1 (\Gamma,v) \ar[r]^{(\vec{x},\sigma)} \ar[d] & S(m,n) \ar[d]^{\eta_{m,m}} \\
H_1 (\Gamma,\Z/m\Z) \ar[r]^-{\varphi_{\gamma}} & \Z/m\Z }
\]

\section{Proof of Theorems~\ref{Thm:MainThm} and~\ref{Thm:Exceptional}}
\label{Sec:Proof}
The goal of this section is to prove Theorems~\ref{Thm:MainThm} and~\ref{Thm:Exceptional}.  Throughout, we assume that $(\Gamma,\gamma)$ is a 2-vertex connected twist graph with $n+1$ vertices, no loops, and with $\varphi_{\gamma}$ surjective.  

Let $v$ be a vertex in $\Gamma$.  For ease of notation, we write $G$ for the image of $\pi_1 (\Gamma,v)$ in $S(m,n)$ and $H$ for $G \cap \mathrm{ker}(\pi)$.  To classify the possible subgroups $G \subseteq S(m,n)$, we begin by classifying the possible subgroups $H \subseteq (\Z/m\Z)^n$.  To that end, we first identify a useful subset of $H$.  The existence of this subset is essentially the only part of the argument that uses the homomorphism from $\pi_1 (\Gamma,v)$ to $S(m,n)$.

\begin{proposition}
\label{Prop:CyclesInKernel}
There exists a set of generators $\{ a_1 , \ldots , a_g \} \subseteq \Z/m\Z$ and vectors $\vec{x}_1 , \ldots, \vec{x}_g \in H$ such that, for each vertex $w \in V(\Gamma)$, $\vec{x}_{i,w}$ is either 0 or $a_i$.  Moreover, there is at least one vertex $w$ such that $\vec{x}_{i,w} = a_i$, and, if $\Gamma$ is not a cycle (possibly with multi-edges), then there is at least one vertex $w$ such that $\vec{x}_{i,w} = 0$.
\end{proposition}

\begin{proof}
Let $p_1 , \ldots , p_g$ be the simple closed paths from Lemma~\ref{Lem:Generators}.  Let
\[
a_i = \varphi_{\gamma} (p_i) = \eta_{m,m} (\vec{x}(p_i),\sigma_{p_i}) = \sum_{w \in V(\Gamma)} \vec{x}(p_i)_w \Mod{m} .
\]
If $w$ is a vertex of $\Gamma$ not contained in $p_i$, then $\vec{x}(p_i)_w = 0$.  Thus, the sum on the right can be taken over vertices in $p_i$.  Since the closed paths $p_i$ generate $\pi_1 (\Gamma,v)$ and $\varphi_{\gamma}$ is surjective, the elements $a_1, \ldots, a_g$ generate $\Z/m\Z$.

Suppose that $p_i$ has length $k_i$.  Then $\sigma_{p_i}$ is a $(k_i-1)$-cycle.  Thus,
\[
(\vec{x}(p_i),\sigma_{p_i})^{k_i-1} = \Big(\sum_{j=0}^{k_i-2} \sigma_{p_i}^j \vec{x}(p_i) \sigma_{p_i}^{-j}, \mathrm{id} \Big) .
\]
Therefore, there exists a vector $\vec{x}_i$ such that $(\vec{x}(p_i),\sigma_{p_i})^{k_i-1} = (\vec{x}_i,\mathrm{id}) \in G \cap \mathrm{ker}(\pi) = H$.  If $w$ is a vertex of $\Gamma$ not contained in $p_i$, then $\vec{x}_{i,w} = 0$.
On the other hand, if $w$ is contained in $p_i$, then
\[
\vec{x}_{i,w} \equiv \sum_{j=1}^{k_i-1} \vec{x}(p_i)_{\sigma_{p_i}^j (w)} \equiv \sum_{u \in p_i} \vec{x}(p_i)_u \equiv a_i \Mod{m}.
\]
Finally, if $\Gamma$ is not a cycle (possibly with multi-edges), then the paths $p_i$ can be chosen to be non-Hamiltonian.  Thus, there exists a vertex $w$ not contained in $p_i$, and the result follows.
\end{proof}

Note that $H \subseteq (\Z/m\Z)^n$ is invariant under the action of $\pi(G)$.  In \cite{Mortimer}, Mortimer studies subspaces of a vector space $k^n$ that are invariant under known 2-transitive subgroups of $S_n$.  If $m$ is prime, we can use Mortimer's result to classify the possibilities for $H$.  More generally, for any integer $a$ dividing $m$, let $G_a = \rho_{m,a} (G)$ and $H_a = \rho_{m,a} (H)$.  We write
\[
\mathcal{C}^{\perp}_{m,a} := \mathrm{ker}(\eta_{m,a}) \cap \mathrm{ker}(\pi) = \Big\{ \vec{x} \in (\Z/m\Z)^n \mbox{ } \vert \mbox{ } \sum_{i=1}^n x_i \equiv 0 \Mod{a} \Big\} .
\]


\begin{proposition}
\label{Prop:KernelModPrime}
If $\Gamma$ is not a cycle (possibly with multi-edges) and $p$ is a prime dividing $m$, then $H_p$ is equal to either $(\Z/p\Z)^n$ or $\mathcal{C}^{\perp}_{p,p}$.
\end{proposition}

\begin{proof}
The group $H_p$ is invariant under the action of $\pi (G)$, and by Theorem~\ref{Thm:Wilson}, $\pi (G)$ is either $S_n$ $(n \geq 3)$, $A_n$ $(n \geq 4)$, or $\mathrm{PGL}(2,5)$ $(n=6)$.  By \cite[Table~1]{Mortimer}, the heart over $\Z/p\Z$ of $\pi (G)$ acting on $(\Z/p\Z)^n$ is simple\footnote{Indeed, the heart over any field $k$ of $\pi (G)$ is simple, except for the special case where $\pi (G)=A_4$.  The heart of $A_4$ over $k$ is reducible if $k$ contains the field $\mathbb{F}_4$, but this does not occur in our case.}.  Note also that a $(\Z/p\Z)$-subspace of $(\Z/p\Z)^n$ is the same thing as a subgroup of $(\Z/p\Z)^n$.  Thus, by \cite[Lemma~2]{Mortimer}, the only $\pi (G)$-invariant subgroups of $(\Z/p\Z)^n$ are $0$, the group of constant vectors $\mathcal{C}_{p,p}$, $\mathcal{C}^{\perp}_{p,p}$, and $(\Z/p\Z)^n$.

We now prove that $H_p$ cannot be 0 or $\mathcal{C}_{p,p}$.  By Proposition~\ref{Prop:CyclesInKernel}, there exists a set of generators $\{ a_1 , \ldots , a_g \} \subseteq \Z/m\Z$ and vectors $\vec{x}_1 , \ldots ,  \vec{x}_g \in H$ such that, for each vertex $w \in V(\Gamma)$, $\vec{x}_{i,w}$ is either 0 or $a_i$.  Moreover, if $a_i \neq 0$, then the vector $\vec{x}_i$ is non-constant.  Since the set $\{ a_1 , \ldots , a_g \}$ generates $\Z/m\Z$, there exists an $i$ such that $a_i$ is not divisible by $p$.  Thus, $\rho_{m,p}(\vec{x}_i )$ is non-constant, and it follows that $H_p$ cannot be $0$ or $\mathcal{C}_{p,p}$. 
\end{proof}

In addition, the only primes $p$ for which $H_p$ may be equal to $\mathcal{C}^{\perp}_{p,p}$ are $p=2$ and 3, and these two cases are mutually exclusive.

\begin{lemma}
\label{Lem:PIs2}
Let $a>1$ be an integer dividing $m$.  If $H = \mathcal{C}^{\perp}_{m,a}$, then either $a=2$ and $\pi (G)$ is $S_n$ or $\mathrm{PGL}(2,5)$, or $a=3$ and $\pi(G)$ is $A_4$.
\end{lemma}

\begin{proof}
By assumption, the map $\eta_{m,a}$ is surjective.  If $H = \mathcal{C}^{\perp}_{m,a}$, then the restriction of $\eta_{m,a}$ to $H$ is identically zero, and thus $\eta_{m,a}$ factors through the quotient $G/H \cong \pi (G)$.  If $\pi (G)$ is $A_n$ for $n \geq 5$, then since $A_n$ is simple, the map $A_n \to \Z/a\Z$ is trivial, contradicting the fact that $\eta_{m,a}$ is surjective.

If $\pi (G) = A_4$, then since all the subgroups of order 3 in $A_4$ are conjugate, no proper normal subgroup can contain an element of order 3.  It follows that the only non-trivial proper normal subgroup of $A_4$ is the Klein 4-group $K$, and the surjective map $A_4 \to \Z/a\Z$ must be the quotient $q : A_4 \to A_4/K \cong \Z/3\Z$.

Similarly, if $\pi (G)$ is $S_n$ or $\mathrm{PGL}(2,5) \cong S_5$, then $a$ must be 2 and the map $S_n \to \Z/a\Z$ must be the sign homomorphism.  To see this, note that every transposition in $S_n$ must map to either the identity or an element of order 2.  Since the transpositions generate $S_n$, it follows that every element of $S_n$ must map to either the identity or an element of order 2.  Since $\Z/a\Z$ contains at most one element of order 2, the surjectivity of $\eta_{m,a}$ implies that $a=2$.  Finally, since the transpositions are all conjugate in $S_n$ and $\Z/2\Z$ is abelian, we see that all the transpositions must map to the same element of $\Z/2\Z$.  It follows that all odd permutations must map to 1 and all even permutations to 0.
\end{proof}

We now classify the possible subgroups $H$.

\begin{proposition}
\label{Prop:KernelClassification}
Let $a$ divide $m$.  If $\Gamma$ is not a cycle (possibly with multi-edges), then either $H_a = (\Z/a\Z)^n$, or $H_a = \mathcal{C}^{\perp}_{a,2}$ and $\pi(G)$ is $S_n$ or $\mathrm{PGL}(2,5)$, or $H_a = \mathcal{C}^{\perp}_{a,3}$ and $\pi(G)$ is $A_4$.
\end{proposition}

\begin{proof}
We will prove this by induction on the number of prime factors of $a$.  The base case, where $a=p$ is prime, follows from Proposition~\ref{Prop:KernelModPrime} and Lemma~\ref{Lem:PIs2}.

We break the inductive step into two cases.  First, consider the case where $a=a_1 a_2$, where $a_1, a_2 > 1$ are relatively prime.  If $\pi(G) = S_n$ or $\pi(G) = \mathrm{PGL}(2,5)$, assume without loss of generality that $a_2$ is odd.  Similarly, if $\pi(G) = A_4$, assume that $a_2$ is not divisible by 3.  By the Chinese Remainder Theorem, $(\Z/a\Z)^n \cong (\Z/a_1 \Z)^n \times (\Z/a_2 \Z)^n$.  By induction, $H_{a_2} = (\Z/a_2\Z)^n$, hence $H_a = \rho_{a,a_1}^{-1} (H_{a_1})$.  By induction, $H_{a_1}$ is either $(\Z/a_1\Z)^n$ or $\mathcal{C}^{\perp}_{a_1,p}$ for $p=2$ or 3.  If $H_{a_1} = (\Z/a_1\Z)^n$, we see that $H_a = (\Z/a\Z)^n$, and if $H_{a_1} = \mathcal{C}^{\perp}_{a_1,p}$, we see that $H_a = \mathcal{C}^{\perp}_{a,p}$.

Second, consider the case where $a=p^{\alpha}$ is a prime power, with $\alpha > 1$.  Consider the short exact sequence
\[
\xymatrix{
0 \ar[r] & (\Z/p\Z)^n \ar[r]^{\mu} & (\Z/p^{\alpha}\Z)^n \ar[r]^{r} & (\Z/p^{\alpha-1}\Z)^n \ar[r] & 0, }
\]
where the map $\mu$ is multiplication by $p^{\alpha-1}$ and the map $r$ is reduction $\Mod{p^{\alpha-1}}$.  By induction, $r(H_a) = H_{p^{\alpha-1}}$ is either $(\Z/p^{\alpha-1}\Z)^n$ or $\mathcal{C}^{\perp}_{p^{\alpha-1},p}$.  Moreover, $\mu^{-1} (H_a)$ is invariant under the action of $\pi (G)$, so by \cite[Lemma~2]{Mortimer} it is either $0$, $\mathcal{C}_{p,p}$, $\mathcal{C}^{\perp}_{p,p}$, or $(\Z/a\Z)^n$.  We consider each of these four possibilities in turn.

If $\mu^{-1}(H_a) = (\Z/a\Z)^n$, then $H_a = r^{-1} (H_{p^{\alpha-1}})$.  If $H_{p^{\alpha-1}} = (\Z/p^{\alpha-1}\Z)^n$, then $H_a = (\Z/a\Z)^n$, and if $H_{p^{\alpha-1}} = \mathcal{C}^{\perp}_{p^{\alpha-1},p}$, then $H_a = \mathcal{C}^{\perp}_{p^{\alpha},p}$.

Now, suppose that $\mu^{-1} (H_a) = \mathcal{C}_{p,p}$, and let $\vec{x} \in H_a$ be an element such that $r(\vec{x}) = (p^{\alpha-2}, (p-1)p^{\alpha-2}, 0, 0, \ldots , 0)$.  Then $p\vec{x}$ is in the image of $\mu$, but not in $\mu(\mathcal{C}_{p,p})$, a contradiction.
 
Similarly, if $\mu^{-1}(H_a) = 0$, then since $r(H_a)$ is nonzero, there exists a nonzero $\vec{x} \in H_a$.  Consider the smallest power $\beta$ such that $p^{\beta}\vec{x}$ is in the image of $\mu$.  Note that such a $\beta$ must exist, since $p^{\alpha-1} \vec{x}$ is in the image of $\mu$ for all $\vec{x}$.  Since $\vec{x} \neq 0$, we see that $p^{\beta}\vec{x} \neq 0$, a contradiction.

Finally, suppose that $\mu^{-1} (H_a) = \mathcal{C}^{\perp}_{p,p}$.  If either $\alpha \geq 3$ or $H_{p^{\alpha-1}} = (\Z/p^{\alpha-1}\Z)^n$, then there exists a vector $\vec{x} \in H_a$ such that $r(\vec{x}) = (p^{\alpha-2},0,0,\ldots,0)$.  Then $p\vec{x}$ is in the image of $\mu$, but not in $\mu(\mathcal{C}^{\perp}_{p,p})$, a contradiction.  It follows that $a=p^2$ and $H_p = \mathcal{C}^{\perp}_{p,p}$.  We now show that this is impossible.

Consider the case where $p=2$ and $\pi(G) = S_n$ or $\mathrm{PGL}(2,5)$.  We will show that $H_4 \subseteq \mathcal{C}^{\perp}_{4,4}$, contradicting Lemma~\ref{Lem:PIs2}.  Assume for contradiction that there exists $\vec{x} \in H_4$ with $\eta_{4,4} (\vec{x}) \not\equiv 0 \Mod{4}$.  Since $H_2 = \mathcal{C}^{\perp}_{2,2}$, we see that $\eta_{4,4} (\vec{x}) \equiv 2 \Mod{4}$.  Let 
\[
\mathcal{O} = \{ w \in V(\Gamma) \smallsetminus \{ v \} \mbox{ } \vert \mbox{ } \vec{x}_w \text{ is odd} \} . 
\]
We reduce to the case where $\mathcal{O}$ is strictly contained in $V(\Gamma) \smallsetminus \{ v \}$.  Since the set $\{ a_1 , \ldots , a_g \}$ generates $\Z/2\Z$, there exists an $i$ such that $a_i$ is odd.  If $\eta_{m,4} (\vec{x}_i) \not\equiv 0 \Mod{4}$, then we let $\vec{x}$ be $\rho_{m,4} (\vec{x}_i)$.  If $\mathcal{O} = V(\Gamma) \smallsetminus \{ v \}$ and $\eta_{m,4} (\vec{x}_i) \equiv 0 \Mod{4}$, then replace $\vec{x}$ with $\vec{x} + \rho_{m,4} (\vec{x}_i)$.
  
Now, let $A \subset V(\Gamma) \smallsetminus \{ v \}$ be a subset containing $\mathcal{O}$ of size $\vert A \vert = k$, with $k$ odd.  (If $\pi (G) = \mathrm{PGL}(2,5)$, choose $k=5$.)  Let $\sigma \in \pi (G)$ be a $k$-cycle that transitively permutes the elements of $A$.  (In the case where $\pi (G) = \mathrm{PGL}(2,5)$, there is a 5-cycle in $\mathrm{PGL}(2,5)$ fixing any given element of $\mathbb{P}^1 (\mathbb{F}_5)$.)  Since $H_4$ is $\pi(G)$-invariant, it contains the vector $\vec{z} = \sum_{i=0}^{k-1} \sigma^i \cdot \vec{x}$.  Note that
\begin{displaymath}
\vec{z}_w = \left\{ \begin{array}{ll}
\sum_{u \in A} \vec{x}_u & \textrm{if $w \in A$,} \\
k \vec{x}_w & \textrm{if $w \notin A$.}
\end{array} \right.
\end{displaymath}
Since $\sum_{u \in A} \vec{x}_u$ is even, $\vec{z}$ is in the image of $\mu$.  On the other hand, we have
\[
\eta_{4,4} (\vec{z}) \equiv  k\Big( \sum_{u \in A} \vec{x}_u + \sum_{u \notin A} \vec{x}_u \Big) \equiv k\eta_{4,4} (\vec{x}) \equiv 2 \Mod{4},
\]
where the last equality holds because $k$ is odd.  Thus, $\vec{z}$ is in the image of $\mu$, but not in $\mu(\mathcal{C}^{\perp}_{2,2})$, a contradiction.

Similarly, consider the case where $p=3$ and $\pi(G) = A_4$.  As above, there exists an $i$ such that $a_i$ is not divisible by 3.  The vector $\rho_{m,9} (\vec{x}_i)$ must have one entry equal to 0 and three entries equal to $a_i \Mod{3}$.  Let $K \subset A_4$ be the Klein 4-group.  Since $H_9$ is $A_4$-invariant, it contains the vector $\vec{z} = \sum_{\sigma \in K} \sigma \cdot \rho_{m,9} (\vec{x}_i) = (3a_i , 3a_i , 3a_i , 3a_i)$.  Thus, $\vec{z}$ is in the image of $\mu$, but not in $\mu(\mathcal{C}^{\perp}_{3,3})$, a contradiction. 
\end{proof}

We now prove the main theorem.

\begin{proof}[Proof of Theorem~\ref{Thm:MainThm}]
If $\Gamma$ is bipartite, then by Theorem~\ref{Thm:Wilson}, $\pi(G) = A_n$.  The only simple 2-vertex connected bipartite graphs on 5 or fewer vertices are the 4-cycle and the graph $\Theta_5$.  By assumption, therefore, $n \geq 5$.  By Proposition~\ref{Prop:KernelClassification}, $H = (\Z/m\Z)^n$.  It follows that
\[
G = \pi^{-1} (A_n) = \Big\{ (\vec{x},\sigma) \in S(m,n) \mbox{ } \vert \mbox{ } \sigma \in A_n \Big\} .
\]

If $\Gamma$ is not bipartite, then by Theorem~\ref{Thm:Wilson}, $\pi(G) = S_n$ with $n \geq 3$.  By Proposition~\ref{Prop:KernelClassification}, either $H = (\Z/m\Z)^n$ or $m$ is even and $H = \mathcal{C}^{\perp}_{m,2}$.  If $H = (\Z/m\Z)^n$, then
\[
G = \pi^{-1} (S_n) = S(m,n) .
\]
If $H = \mathcal{C}^{\perp}_{m,2}$, then as in the proof of Lemma~\ref{Lem:PIs2}, the map $\eta_{m,2} : G \to \Z/2\Z$ factors through the quotient $G/H \cong S_n$, and the map $S_n \to \Z/2\Z$ is the sign homomorphism.  It follows that
\[
G = \Big\{ (\vec{x},\sigma) \in S(m,n) \mbox{ } \vert \mbox{ } \eta_{m,2} (\vec{x},\sigma) \equiv \mathrm{sign} (\sigma) \Mod{2} \Big\} .
\]

It remains to classify which non-bipartite twist graphs correspond to each of the two subgroups.  If $p$ is a simple closed path, then $\eta_{m,2} (\vec{x(p)},\sigma_p) \equiv \mathrm{sign} (\sigma_p) \Mod{2}$ if and only if $p$ has an even number of edges $e$ with $\gamma_e$ even.  Since $\pi_1 (\Gamma,v)$ is generated by simple closed paths, it follows that $G \neq S(m,n)$ if and only if $(\Gamma,\gamma)$ is twist bipartite.
\end{proof}

The exceptional cases are very similar.

\begin{proof}[Proof of Theorem~\ref{Thm:Exceptional}]
If $\Gamma = \Theta_7$, the argument is nearly identical to the non-bipartite case above.  By Theorem~\ref{Thm:Wilson}, $\pi(G) =  \mathrm{PGL}(2,5)$, and by Proposition~\ref{Prop:KernelClassification}, either $H = (\Z/m\Z)^6$ or $m$ is even and $H = \mathcal{C}^{\perp}_{m,2}$.  If $H = (\Z/m\Z)^6$, then $G = \pi^{-1} (\mathrm{PGL}(2,5))$, and if $H = \mathcal{C}^{\perp}_{m,2}$, then as in the proof of Lemma~\ref{Lem:PIs2}, the map $\eta_{m,2} : G \to \Z/2\Z$ factors through the quotient $G/H \cong \mathrm{PGL}(2,5) \cong S_5$, and the map $S_5 \to \Z/2\Z$ is the sign homomorphism.

If $\Gamma = \Theta_5$, then by Proposition~\ref{Prop:KernelClassification}, either $H = (\Z/m\Z)^4$ or $m$ is divisible by 3 and $H = \mathcal{C}^{\perp}_{m,3}$.  If $H = (\Z/m\Z)^4$, then again 
\[
G = \pi^{-1} (A_4) = \Big\{ (\vec{x},\sigma) \in S(m,4) \mbox{ } \vert \mbox{ } \sigma \in A_4 \Big\} .
\]
If $H = \mathcal{C}^{\perp}_{m,3}$, then as in the proof of Lemma~\ref{Lem:PIs2}, the map $\eta_{m,3} : G \to \Z/3\Z$ factors through the quotient $G/H \cong A_4$, and the map $A_4 \to \Z/3\Z$ is the quotient of $A_4$ by the Klein 4-group $K$.  It follows that
\[
G = \Big\{ (\vec{x},\sigma) \in S(m,4) \mbox{ } \vert \mbox{ } \sigma \in A_4 , \eta_{m,3} (\vec{x},\sigma) \equiv q(\sigma) \Mod{3} \Big\}.
\]

It remains to classify which 1-chains $\gamma$ on $\Theta_4$ correspond to each of these two subgroups.  The fundamental group of $\Theta_4$ is generated by the two simple closed paths $p = e_1 e_2 \overline{e}_6 \overline{e}_5$ and $q = e_3 e_4 \overline{e}_6 \overline{e}_5$.  Because these two paths generate $\pi_1 (\Gamma, v)$, $G$ is contained in a subgroup if and only if $(\vec{x}(p),\sigma_p)$ and $(\vec{x}(q),\sigma_q)$ are contained in that subgroup.  It follows that $G \neq \pi^{-1} (A_4)$ if and only if the two classes 
\[
(\gamma_{e_1} + \gamma_{e_2}) - (\gamma_{e_5} + \gamma_{e_6}), \mbox{ } (\gamma_{e_3} + \gamma_{e_4}) - (\gamma_{e_5} + \gamma_{e_6})
\]
are nonzero and distinct $\Mod{3}$.  Equivalently, $G \neq \pi^{-1} (A_4)$ if and only if 
\[
(\gamma_{e_1} + \gamma_{e_2}),  \mbox{ } (\gamma_{e_3} + \gamma_{e_4}),  \mbox{ and } (\gamma_{e_5} + \gamma_{e_6})
\]
are distinct $\Mod{3}$.

Finally, note that for any three integers $x$, $y$, and $z$, we have $x+y+z \equiv 0 \Mod{3}$ if and only if either $x$, $y$, and $z$ are all equivalent $\Mod{3}$, or distinct $\Mod{3}$.  But if
\[
(\gamma_{e_1} + \gamma_{e_2}) \equiv (\gamma_{e_3} + \gamma_{e_4}) \equiv (\gamma_{e_5} + \gamma_{e_6}) \Mod{3},
\]
then $\varphi_{\gamma}$ is not surjective.  It follows that $G \neq \pi^{-1} (A_4)$ if and only if $\sum_{i=1}^6 \gamma_{e_i} \equiv 0 \Mod{3}$.
\end{proof}

\section*{Acknowledgments}
We would like to thank Henry Segerman for comments on an early draft of this article, and for providing us with high quality photos.  This research was conducted as a project with the University of Kentucky Math Lab, supported by NSF DMS-2054135.



\end{document}